\documentclass[11pt]{article}

\usepackage[dvips]{graphicx}
\usepackage{pictex}

\usepackage{amsmath,amsthm,amsfonts}   % These are to write mathematics
\usepackage{amscd}
\usepackage{amsfonts}
\usepackage{amssymb}
\usepackage{color}      % Need the color package
\usepackage{epsfig}
\usepackage{graphicx}           % Graphics package

\newtheorem{lemma}{Lemma}[section]
\newtheorem{theorem}[lemma]{Theorem}
\newtheorem{proposition}[lemma]{Proposition}
\newtheorem{corollary}[lemma]{Corollary}

\newtheorem{dnt}[lemma]{Definition}

\newtheorem{exm}[lemma]{Example}
\newtheorem{rem}[lemma]{Remark}

\newenvironment{prop}{\begin{proposition}}{\end{proposition}}

\newenvironment{cor}{\begin{corollary}}{\end{corollary}}

\begin{document}

\title{Domination in Functigraphs}

\author{
Linda Eroh$^1$ \hspace*{1in} Ralucca Gera$^2$ \hspace*{1in} Cong X. Kang$^3$ \\
Craig Larson$^4$  \hspace*{1in} Eunjeong Yi$^5$\\
$^1$\small Department of Mathematics, University of Wisconsin Oshkosh \\
\small Oshkosh, WI 54901; {\small\em eroh@uwosh.edu}\\
$^2$\small Department of Applied Mathematics, Naval Postgraduate School,\\
 \small Monterey, CA $93943$;  {\small\em rgera@nps.edu}\\
 $^{3,5}$\small Department of General Academics, Texas A\&M University at Galveston\\
 \small Galveston, TX $77553$; {\small\em kangc@tamug.edu$^3$, yie@tamug.edu$^5$}\\
 $^4$\small Department of Mathematics and Applied Mathematics, Virginia Commonwealth University \\
 \small Richmond, VA 23284; {\small\em clarson@vcu.edu}
 }

\maketitle

%\title{Domination in Functigraphs}

\begin{abstract}
Let $G_1$ and $G_2$ be disjoint copies of a graph $G$, and let $f:
V(G_1) \rightarrow V(G_2)$ be a function. Then a
\emph{functigraph} $C(G, f)=(V, E)$ has the vertex set $V=V(G_1)
\cup V(G_2)$ and the edge set $E=E(G_1) \cup E(G_2)$ $\cup \ \{uv
 \ | \ u \in V(G_1), v \in V(G_2), v=f(u)\}$. A functigraph is a
generalization of a \emph{permutation graph} (also known as a
\emph{generalized prism}) in the sense of Chartrand and Harary. In
this paper, we study domination in functigraphs. Let $\gamma(G)$ denote the domination number of $G$. It is readily
seen that $\gamma(G) \le \gamma(C(G,f)) \le 2 \gamma(G)$. We
investigate for graphs generally, and for cycles in great detail,
the functions which achieve the upper and lower bounds, as well as
the realization of the intermediate values.
\end{abstract}



%%%%%%%%%%%%%%%%%%%%%%%%%%%%%%%%%%%%%%%%%%%%%%%%%%
%%%%%%%%%%%%%%%%%%%%%%%%%%%%%%%%%%%%%%%%%%%%%%%%%%

\section{Introduction and Definitions}

Throughout this paper, $G = (V(G),E(G))$ stands for a finite,
undirected, simple and connected graph with order $|V(G)|$ and
size $|E(G)|$. A set $D \subseteq V(G)$ is a \emph{dominating set}
of $G$ if for every vertex $v \in V(G) \setminus D$, there exists a vertex
$u \in D$ such that $v$ and $u$ are adjacent. The
\textit{domination number} of a graph $G$, denoted by $\gamma(G)$,
is the minimum of the cardinalities of all dominating sets of $G$.
For earlier discussions on domination in graphs, see \cite{B1, B2,
EJ, Ore}. For further reading on domination, refer to \cite{Dom2}
and \cite{Dom1}.\\

For any vertex $v \in V(G)$, the \emph{open neighborhood} of $v$
in $G$, denoted by $N_G(v)$, is the set of all vertices adjacent
to $v$ in $G$. The \emph{closed neighborhood} of $v$, denoted by
$N_G[v]$, is the set $N_G(v) \cup \{v\}$. Throughout the paper, we
denote by $N(v)$ (resp., $N[v]$) the open (resp., closed)
neighborhood of $v$ in $C(G,f)$. The maximum degree of $G$ is denoted by $\Delta(G)$. 
For a given graph $G$ and $S \subseteq V(G)$, we denote by $\langle S \rangle$ the subgraph induced by $S$. 
Refer to \cite{CZ} for additional graph theory terminology.\\

Chartrand and Harary studied planar permutation graphs in
\cite{CH}. Hedetniemi introduced two graphs (not necessarily identical copies) with a function relation between them; he called the resulting object a ``function graph" \cite{H}. Independently, D\"{o}rfler introduced a ``mapping graph", which consists of two disjoint identical copies of a graph and additional edges between the two vertex sets specified by a function \cite{WD}. Later, an extension of permutation graphs, called {\it
functigraph}, was rediscovered and studied in \cite{functi}. In the current paper, we
study domination in functigraphs. We recall the definition of a
functigraph in \cite{functi}.
\begin{dnt}
Let $G_1$ and $G_2$ be two disjoint copies of a graph $G$,
 and let $f$ be a function from $V(G_1)$ to $V(G_2)$. Then a
functigraph $C(G, f)$ has the vertex set
$$V(C(G, f)) = V(G_1) \cup V(G_2),$$
and the edge set
$$E(C(G, f)) = E(G_1) \cup E(G_2) \cup \{ uv \ | \  u \in V(G_1), v\in V(G_2), v=f(u)\}.$$
\end{dnt}

Throughout the paper, $V(G_1)$ denotes the {\em domain} of a
function $f$; $V(G_2)$ denotes the {\em  codomain} of $f$;
$Range(f)$ denotes the {\em range} of $f$. For a set $S \subseteq
V(G_2)$, we denote by $f^{-1}(S)$ the set of all pre-images of the
elements of $S$; i.e., $f^{-1}(S) = \{v \in V(G_1) : f(v) \in
S\}$. Also, $C_n$ denotes a cycle of length $n \ge 3$, and $id$
denotes the identity function. Let $V(G_1)=\{u_1, u_2, \ldots,
u_n\}$ and $V(G_2)=\{v_1, v_2, \ldots, v_n\}$. For simplicity, we
sometimes refer to each vertex of the graph $G_1$ (resp., $G_2$)
by the index $i$ (resp., $i'$) of its label $u_i$ (resp., $v_i$)
for $1 \le i, i' \le n$. When $G=C_n$, we assume that the vertices
of $G_1$ and $G_2$ are labeled cyclically. It is readily seen that
$\gamma(G) \le \gamma(C(G,f)) \le 2 \gamma(G)$. We study the
domination of $C(C_n,f)$ in great detail: for $n \equiv 0 \pmod
3$, we characterize the domination number for an infinite class of
functions and state conditions under which the upper bound is not
achieved; for $n \equiv 1,2 \pmod 3$, we prove that, for any
function $f$, the domination number of $C(C_n,f)$ is strictly less
than $2 \gamma(C_n)$. These results extend and generalize
a result by Burger, Mynhardt, and Weakley in \cite{Mynhardt}.\\

Domination number on permutation graphs (generalized prisms) has
been extensively investigated in a great many articles, among
these are \cite{thesis, Benecke, Burger, Mynhardt, Hartnell}; the
present paper primarily deepens -- and secondarily broadens -- the
current state of knowledge.

%%%%%%%%%%%%%%%%%%%%%%%%%%%%%%%%%%%%%%%%%%%%%%%%%%%%%%%%%%%%%%%%%%%%%%%%%%%%%%%%%%%%%%%%%%%%%%%%%%%%%%%%%%%
%%%%%%%%%%%%%%%%%%%%%%%%%%%%%%%%%%%%%%%%%%%%%%%%%%%%%%%%%%%%%%%%%%%%%%%%%%%%%%%%%%%%%%%%%%%%%%%%%%%%%%%%%%

\section{Domination Number of Functigraphs}

First we consider the lower and upper bounds of the domination number
of $C(G, f)$.

\begin{prop}  \label{dom_bound}
For any graph $G$, $\gamma(G) \le \gamma(C(G, f)) \le 2
\gamma(G)$.
\end{prop}

\begin{proof}
Let $D$ be a dominating set of $G$. Since a copy of $D$ in $G_1$
together with a copy of $D$ in $G_2$ form a dominating set of
$C(G, f)$ for any function $f$, the upper bound follows. For the
lower bound, assume there is a dominating set $D$ of $C(G, f)$ such
that $|D| < \gamma(G)$. Let $D_1 = D \cap V(G_1) \neq \emptyset$
and $D_2 = D \cap V(G_2) \neq \emptyset$, with $D_1 \cup D_2 = D$.
Now, for each $x \in D_1$, $x$ dominates exactly one vertex in
$G_2$, namely $f(x)$. And so $D_2 \cup \{f(x) \ | \ x \in D_1\}$
is a dominating set of $G_2$ of cardinality less than or equal to
$|D|$, but $|D| < \gamma(G_2)$
--- a contradiction. \hfill
\end{proof}

Next we consider realization results for an arbitrary graph $G$.

\vspace{.1in}

\begin{theorem} \label{realization}
For any pair of integers $a, b$ such that $1 \le a \le b \le 2a$,
there is a connected graph $G$ for which $\gamma(G) = a$ and
$\gamma(C(G, f))=b$ for some function $f$.
\end{theorem}

\begin{proof}
Let the star $S_i\cong K_{1,4}$ have center $c_i$ for $1\leq i\leq a$.
Let $G$ be a chain of $a$ stars; i.e., the disjoint union of $a$
stars such that the centers are connected to form a path of length
$a$ (and no other additional edges) -- see Figure \ref{fig1}. Label the
stars in the chain of the domain $G_1$ by $S_1, S_2, \ldots, S_a$
and label their centers by $c_1, c_2, \ldots, c_a$, respectively.
Likewise, label the stars in the chain of the codomain $G_2$ by
$S_1', S_2', \ldots, S_a'$ and label their centers by $c_1', c_2',
\dots, c_a'$, respectively. More generally, denote by $v'$ the
vertex in $G_2$ corresponding to an arbitrary $v$ in $G_1$.\\

We define $a+1$ functions from $G_1$ to $G_2$ as follows. Let
$f_0$ be the ``identity function"; i.e., $f_0(v)=v'$. For each $i$
from $1$ to $a$, let $f_i$ be the function which collapses $S_1$
through $S_i$ to $c_1'$ through $c_i'$, respectively, and which
acts as the ``identity" on the remaining vertices: $f_i(S_j)=c_j'$
for $1\leq j \leq i$ and $f_i(v)=v'$ for $v\notin
\bigcup_{\scriptstyle 1 \leq j \leq i}V(S_j)$. (See
Figure 1.) Notice $\gamma(G)=a$.\\

\begin{figure}[htbp]
\begin{center}\label{fig1}
\scalebox{0.45}{\input{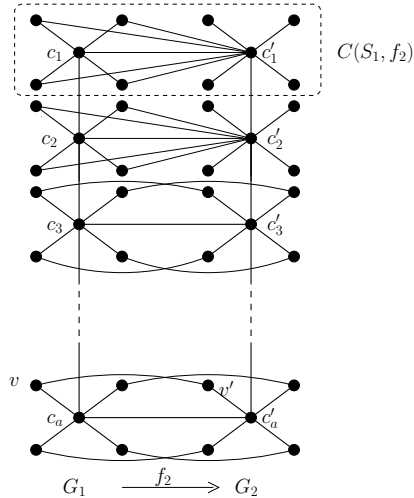}}\caption{Realization Graphs}
\end{center}
\end{figure}

Claim: $\gamma(C(G,f_i))=2a-i$ for $0\leq i\leq a$.\\

First, $\gamma(C(G,f_a))=a$ because $D_a=\{c_1', ..., c_a'\}$
clearly dominates $C(G,f_a)$.\\

Second, consider $C(G,f_0)$. $D_0=\{c_1, ..., c_a, c_1', ...,
c_a'\}$, the set of centers in $G_1$ or $G_2$, is a dominating
set; so $\gamma(C(G,f_0))\leq 2a$ as noted earlier. It suffices to
show that $\gamma(C(G,f_0))\geq 2a$. It is clear that a dominating
set $D$ consisting only of the centers must have size $2a$ --- for
a pendant to be dominated, its neighboring center must be in $D$.
We need to check that the replacement of centers by some (former)
pendants (of $G_1$ or $G_2$) will only result in a dominating set
$D'$ such that $|D'|>|D_0|$. It suffices to check $C(S_i,f_0)$ at
each $i$, a subgraph of $C(G,f_0)$ -- since pendant domination is
a local question: the closed neighborhood of each pendant of
$C(S_i,f_0)$ is contained within $C(S_i,f_0)$. It is easy to see
that the unique minimum dominating set of $C(S_i,f_0)$ consists
of the two centers $c_i$ and $c_i'$.\\

Finally, the set $D_i=\{c_{i+1}, ..., c_a, c_1', ..., c_a'\}$ is a
minimum dominating set of $C(G,f_i)$: In relation to $C(G,f_0)$,
the subset $\{c_1, ..., c_i\}$ of $D_0$ is not needed since
the set $\{c_1', ..., c_i'\}$ dominates $\bigcup_{\scriptstyle 1 \leq j
\leq i}V(S_j)$ in $C(G,f_i)$. The local nature of pendant
domination and the fact that $f_i|_{S_j}=f_0|_{S_j}$ for $j>i$ ensure
that $D_i$ has minimum cardinality.~\hfill
\end{proof}

%%%%%%%%%%%%%%%%%%%%%%%%%%%%%%%%%%%%%%%%%%%%%%%%%%%%%%%%%%%
%%%%%%%%%%%%%%%%%%%%%%%%%%%%%%%%%%%%%%%%%%%%%%%%%%%%%%%%%%%

\section{Characterization of Lower Bound}

We now present a characterization for $\gamma(C(G, f))=\gamma(G)$, in analogy with what was done for permutation-fixers in
\cite{Burger}.\\

\begin{theorem}\label{gen}
Let $G_1$ and $G_2$ be two copies of a graph $G$ in $C(G, f)$.
Then $\gamma(G)=\gamma(C(G,f))$ if, and only if, there are sets
$D_1\subseteq V(G_1)$ and $D_2\subseteq V(G_2)$ satisfying the
following conditions:
\begin{enumerate}
\item $D_1$ dominates $V(G_1) \setminus f^{-1}(D_2)$, \item $D_2$
dominates $V(G_2)\setminus f(D_1)$, \item $D_2 \cup f(D_1)$ is a
minimum dominating set of $G_2$, \item $|D_1|=|f(D_1)|$, \item
$D_2 \cap f(D_1)=\emptyset$, and \item $D_1 \cap
f^{-1}(D_2)=\emptyset$.
\end{enumerate}
\end{theorem}

\begin{proof}
($\Longleftarrow$) Suppose there are sets $D_1\subseteq V(G_1)$
and $D_2\subseteq V(G_2)$ satisfying the specified conditions.
Clearly $D_1\cup D_2$ is a dominating set of $C(G,f)$. By
assumption, $D_2\cup f(D_1)$ is a minimum dominating set of $G_2$.
Since $|D_1| = |f(D_1)|$ and $D_2 \cap f(D_1) = \emptyset$,
$\gamma(G)=\gamma(G_2)=|D_2|+|f(D_1)|=|D_2|+|D_1|$. Since
$\gamma(G) \le \gamma(C(G,f)) \le |D_1|+|D_2|=\gamma(G)$, it
follows that $\gamma(G)=\gamma(C(G,f))$.\\

($\Longrightarrow$) Let $D$ be any minimum dominating set of
$C(G,f)$. Suppose then that $\gamma(G)=\gamma(C(G,f))$ such that
$D_1 = D \cap V(G_1)$ and $D_2 = D \cap V(G_2)$.  So
$\gamma(C(G,f))=|D_1|+|D_2|$. Note that the only vertices in $G_2$
that are dominated by $D_1$ are the vertices in $f(D_1)$ and the
only vertices in $G_1$ that are dominated by $D_2$ are the
vertices in $f^{-1}(D_2)$. Since $D$ is a dominating set of
$C(G,f)$, $D_2$ must dominate every vertex in $V(G_2)\setminus
f(D_1)$, and $D_1$ must dominate every vertex in $V(G_1)\setminus
f^{-1}(D_2)$.\\

Clearly $D_2 \cup f(D_1)$ is a dominating set of $G_2$. Note that
$|D_1|\geq |f(D_1)|$. So $\gamma(G)=\gamma(C(G,f))=|D_1|+|D_2|\geq
|D_2|+|f(D_1)|\geq \gamma(G_2)=\gamma(G)$. But then these terms
must all be equal. In particular, $|D_1|=|f(D_1)|$ and $D_2\cup
f(D_1)$ is a minimum dominating set of $G_2$. Furthermore, $D_2
\cap f(D_1)=\emptyset$, else $D_2\cup f(D_1)$ is a dominating set
of $G_2$ with fewer than $\gamma(G_2)$ vertices. Finally, suppose
there is a vertex $v\in D_1\cap f^{-1}(D_2)$. So $v\in D_1$ and
$v\in f^{-1}(D_2)$. But then $f(v)\in f(D_1)$ and $f(v)\in D_2$.
But $f(D_1)$ and $D_2$ are disjoint. So,  $D_1 \cap
f^{-1}(D_2)=\emptyset$. \hfill
\end{proof}

It is known that for cycles $C_n$ ($n\geq 3$),
$\gamma(C_n)=\lceil\frac{n}{3}\rceil$. We now apply Theorem
\ref{gen} to characterize the lower bound of $\gamma(C(C_n, f))$.

\begin{theorem} \label{LB cycle}
For the cycle $C_n$ ($n \ge 3$), let $G_1$ and $G_2$ be copies of
$C_n$. Then $\gamma(C_n)=\gamma(C(C_n,f))$ if, and only if, there
is a minimum dominating set $D=D_1 \cup D_2$ of $C(C_n,f)$ such
that either:
\begin{enumerate}
\item $D_1 = \emptyset$ and $D_2$ is a minimum dominating set of
$G_2$ and $Range(f)\subseteq D_2$, or \item $n\equiv 1$ (mod $3$),
$D_2$ is a minimum dominating set for $\langle V(G_2) \setminus \{v\} \rangle$,
$D_1=\{w\}$, $f(w)=v$, and $f(V(G_1)\setminus N[w])\subseteq D_2$.
\end{enumerate}
\end{theorem}

\begin{figure}[htbp]
\begin{center}\label{1mod3}
\scalebox{0.5}{\input{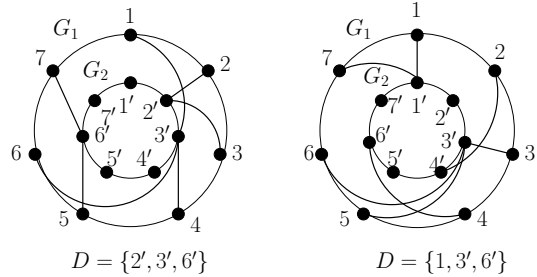}} \caption{Examples of
$\gamma(C(C_n,f))=\gamma(C_n)$ for $n \equiv 1$ (mod 3)}
\end{center}
\end{figure}

\begin{proof}
($\Longleftarrow$) Suppose that there is a minimum dominating set
$D$ of $C(C_n,f)$ satisfying the specified conditions. So
$\gamma(C(C_n,f))=|D|=|D_1|+|D_2|$. If $D_2 \subseteq V(G_2)$ is a
minimum dominating set of $C_n$ and $Range(f)\subseteq D_2$, then
$D_1=\emptyset$. So $\gamma(C_n)=|D_2|=\lceil \frac{n}{3} \rceil$.
Furthermore $\gamma(C(C_n,f))=|D|=|D_1|+|D_2|=0+ \gamma(G_2)$.\\

Suppose $n\equiv 1$ (mod $3$), $D_2$ dominates all but one vertex
$v$ of $G_2$, $D_1=\{w\}$, $f(w)=v$,
 and $f(V(G_1)\setminus N[w])\subseteq D_2$.
 Note that, since $n\equiv 1$ (mod $3$), $n=3k+1$, for some positive integer $k$,
 and $\lceil\frac{n}{3}\rceil=k+1$. By assumption, $\gamma(C(C_n,f))=|D|=|D_1|+|D_2|=1+|D_2|$.
 Since $\gamma(C_n)=k+1$,
 it remains to show that $\gamma(C(C_n,f))=k+1$,
 which is equivalent to showing that $|D_2|=k$.
Since $D_2$ is a minimum dominating set for $\langle V(G_2) \setminus \{v\} \rangle$ and $\langle V(G_2) \setminus \{v\} \rangle$ has domination number $k$, $|D_2|=k$.\\

($\Longrightarrow$) Now suppose that
$\gamma(C_n)=\gamma(C(C_n,f))=\lceil\frac{n}{3}\rceil$. Let $D$ be
a minimum dominating set satisfying the conditions of Theorem
\ref{gen}. There are three cases to consider: $n\equiv 0$ (mod
$3$), $n\equiv 1$ (mod $3$), and $n\equiv 2$ (mod $3$). In each
case, Theorem \ref{gen} implies that $D_2\cup f(D_1)$ is a minimum
dominating set of $G_2$ and $|D_1|=|f(D_1)|$. Since $f(D_1)$ must
include all the vertices not dominated by $D_2$, it follows that
$D$ must contain at least $|D_2|+(n-3|D_2|)=n-2|D_2|$ vertices.\\

If $n\equiv 0$ (mod $3$), then $n=3k$ for some positive integer
$k$ and $\lceil\frac{n}{3}\rceil=k$. Note that $D_2$ dominates at most
$3|D_2|$ vertices in $G_2$. There are at least $n-3|D_2|$ vertices in
$G_2$ which are not dominated by $D_2$. If $|D_2|<k$ then
$\gamma(C(C_n,f))=|D|\geq n-2|D_2|>n-2k=3k-2k=k$, contradicting the
assumption that $\gamma(C(C_n,f))=k$. So $|D_2|=k$. This implies
$D_1=\emptyset$. And this, in turn, implies that $D_2$ must
dominate all the vertices in $G_1$. So $Range(f)\subseteq D_2$.\\

In the remaining two cases, where $n\equiv 1$ or $n\equiv 2$ (mod
$3$), then $n=3k+1$ or $n=3k+2$, respectively, for some positive
integer $k$ and $\gamma(C_n)=\lceil\frac{n}{3}\rceil=k+1$. From
Theorem \ref{gen} it follows that $D_2\cup f(D_1)$ is a minimum
dominating set of $G_2$. Since $D_2$ dominates at most $3|D_2|$
vertices in $G_2$, $D_1$ must dominate at least $n-3|D_2|$
vertices in $G_2$. If $|D_2|<k$, then $\gamma(C(C_n,f))=|D|\geq
n-2|D_2|>n-2k=(3k+1)-2k=k+1$, contradicting the assumption that
$\gamma(C(C_n,f))=k+1$. So $|D_2|\geq k$. Since $|D|=k+1$,
$|D_2|\leq k+1$. If $|D_2|=k+1$, then $D_1=\emptyset$,
$f(D_1)=\emptyset$ and $D_2\cup f(D_1)=D_2$ is a minimum
dominating set of $G_2$. Since $D$ is a dominating set of
$C(C_n,f)$, it follows that $D_2$ must also dominate all the
vertices in $D_1$ and, thus, $Range(f)\subseteq D_2$.\\

Let $n\equiv 1 \pmod 3$. If $|D_2|=k$, then there is at least one
vertex in $G_2$ not dominated by $D_2$. If there are $c>1$
vertices not dominated by $D_2$ then these vertices are a subset
of $f(D_1)$ and Theorem \ref{gen} guarantees that
$|D_1|=|f(D_1)|\geq c$ and, thus, $\gamma(C(C_n,f))\geq k+c>k+1$,
contradicting our assumption. So $c=1$. There is only one vertex
$v\in V(G_2)$ which is not dominated by $D_2$. $D_1$ can only
contain a single vertex $w$ (or $|D|$ will again be too large) and
$f(w)=v$. Since $w$ dominates $N[w]$ in $G_1$, it follows that
$D_2$ must dominate $V(G_1)\setminus N[w]$. So $f(V(G_1)\setminus
N[w])\subseteq D_2$.\\

Let $n\equiv 2 \pmod 3$. If $|D_2|=k$, then there are at least two
vertices in $G_2$ not dominated by $D_2$. But then these vertices
must be a subset of $f(D_1)$ and $|f(D_1)|\geq 2$. Since
$|D_1|=|f(D_1)|$, $|D_1|\geq 2$. But then
$k+1=\gamma(C(G,f))=|D|=|D_1|+|D_2|\geq 2+k$, which is a
contradiction. So $|D_2|=k+1$. \hfill
\end{proof}

Next we consider the domination number of $C(C_3,f)$.

\begin{lemma}\label{cycle 3}
Let $G_1$ and $G_2$ be two copies of $C_3$. Then $\gamma(C(C_3,
f))=2 \gamma(C_3)$ if and only if $f$ is not a constant function.
\end{lemma}

\begin{proof}
($\Longleftarrow$) Suppose that $f$ is not a constant function.
Then, for each vertex $v \in V(C(C_3,f))$, $\deg(v) \le 4$ and
hence $N[v] \varsubsetneq V(C(C_3, f))$. Thus $\gamma(C(C_3, f))
\ge 2$. Since there exists a dominating set consisting of one
vertex from each of $G_1$ and $G_2$, $\gamma (C(C_3, f))=2$.\\

($\Longrightarrow$) Suppose that $f$ is a constant
function, say $f(w)=a$ for some $a \in V(G_2)$ and for all $w \in
V(G_1)$. Then $N[a]=V(C(C_3, f))$, and thus $\gamma(C(C_3,
f))=1=\gamma(C_3)$. \hfill
\end{proof}

As an immediate consequence of Theorem \ref{LB cycle} and Lemma
\ref{cycle 3}, we have the following.

\begin{cor} \label{cor on permutation}
There is no permutation $f$ such that $\gamma(C(C_n,
f))=\gamma(C_n)$ for $n=3$ or $n \ge 5.$
\end{cor}

Now we consider $C(G, f)$ when $G=C_n$ ($n \ge 3$) and $f$ is the
identity function.

\begin{theorem} \label{C_n identity}
Let $G_1$ and $G_2$ be two copies of the cycle $C_n$ for $n \ge
3$. Then
\begin{displaymath}
\gamma(C(C_n,id)) = \left\{ \begin{array}{lll}
  \lceil\frac{n}{2}\rceil & \mbox{ if } n \not\equiv 2 \pmod 4, \\
  \frac{n}{2}+1 & \mbox{ if } n \equiv 2 \pmod 4 .
  \end{array}
  \right.
\end{displaymath}
\end{theorem}

\begin{proof}
Since $C(C_n, id)$ is 3-regular, each vertex in $C(C_n, id)$ can
dominate 4 vertices. We consider four cases.\\

\emph{Case 1. $n=4k$:} Since $|V(C(C_n, id))|=8k$, we have
$\gamma(C(C_n, id)) \ge \lceil\frac{8k}{4}\rceil=2k$. Since
$\cup_{j=0}^{k-1} \{4j+1, (4j+3)'\}$ is a dominating set of $C(C_n, id)$ with cardinality $2k$, we conclude that
$\gamma(C(C_n, id))=2k=\lceil\frac{n}{2}\rceil$.\\

\emph{Case 2. $n=4k+1$:} Since $|V(C(C_n, id))|=2(4k+1)=8k+2$, we
have $\gamma(C(C_n, id)) \ge \lceil \frac{8k+2}{4}\rceil=2k+1$.
Since $(\cup_{j=0}^{k} \{4j+1\}) \bigcup(\cup_{i=0}^{k-1}
\{(4i+3)'\})$ is a dominating set of $C(C_n, id)$ with
cardinality $2k+1$, we have $\gamma(C(C_n,
id))=2k+1=\lceil\frac{n}{2}\rceil$.\\

\emph{Case 3. $n=4k+2$:} Notice that $(\cup_{j=0}^{k}\{4j+1\})
\bigcup (\cup_{i=0}^{k-1}\{(4i+3)'\}) \bigcup \{(4k+2)'\}$ is a
dominating set of $C(C_n, id)$ with cardinality
$2k+2=\frac{n}{2}+1$; thus $\gamma(C(C_n,id))\leq 2k+2$. Since
$|V(C(C_n, id))|=2(4k+2)=8k+4$, $\gamma(C(C_n, id)) \ge
\lceil\frac{8k+4}{4}\rceil=2k+1$; indeed, $\gamma(C(C_n,
id))=2k+1$ only if every vertex is dominated by exactly one vertex
of a dominating set; i.e., no double domination is allowed.
However, we show that there must exist a doubly-dominated vertex
for any dominating set by the following \emph{descent} argument:
Let the graph $A_0$ be $P_{4k+3}\times K_2$ where the bottom row
is labeled $1,2,\ldots,4k+2,1$ and the top row is labeled
$1',2',\ldots,(4k+2)',1'$; note that $C(C_n,id)$ is obtained by
identifying the two end-edges each with end-vertices labeled $1$
and $1'$. Without loss of generality, choose $1'$ to be in
a dominating set $D$. For each vertex to be singly dominated, we
delete vertices $1'(s), 1(s), 2', \mbox{ and } (4k+2)'$, as well
as their incident edges, to obtain a derived graph $A_1$. In
$A_1$, vertices $2$ and $4k+2$ are end-vertices and neither may
belong to $D$ as each only dominates two vertices in $A_1$. This
forces support vertices $3$ and $4k+1$ in $A_1$ to be in $D$.
Deleting vertices $2, 3, 3', 4, 4k+2, 4k+1, (4k+1)', \mbox{ and }
4k$ and incident edges results in the second derived graph $A_2$.
After $k$ iterations, $A_k$ is the extension of $P_3\times P_2$ by
two leaves at both ends of either the top or the bottom row (see
Figure \ref{descent}); $A_k$, which has eight vertices, clearly
requires three vertices to be dominated.
\begin{figure}[htbp]
\begin{center}
\scalebox{0.45}{\input{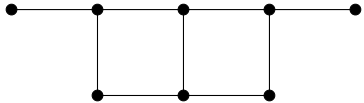}} \caption{$A_k$ in the
$n=4k+2$ case}\label{descent}
\end{center}
\end{figure}
Thus, we conclude that $\gamma(C(C_n, id))=2k+2=\frac{n}{2}+1$.\\

\emph{Case 4. $n=4k+3$:} Since $|V(C(C_n, id))|=2(4k+3)=8k+6$, we
have $\gamma(C(C_n, id)) \ge \lceil\frac{8k+6}{4}\rceil=2k+2$.
Since $\cup_{j=0}^{k} \{4j+1, (4j+3)'\}$ is a dominating set of $C(C_n, id)$ with cardinality $2k+2$, we conclude that
$\gamma(C(C_n, id)=2k+2=\lceil\frac{n}{2}\rceil$. \hfill
\end{proof}

As a consequence of Theorem \ref{C_n identity}, we have the
following result.

\begin{cor}\label{id}
\begin{enumerate}
\item $\gamma(C(C_n, id))=\gamma(C_n)$ if and only if $n=4$. \item
$\gamma(C(C_n, id))=2\gamma(C_n)$ if and only if $n=3$ or $n=6$.
\end{enumerate}
\end{cor}

By Corollary \ref{cor on permutation} and Theorem \ref{C_n
identity}, we have the following result.

\begin{prop}
For a permutation $f$, $\gamma(C(C_n, f))=\gamma(C_n)$ if and only
if $C(C_n, f) \cong C(C_4, id)$.
\end{prop}

\begin{proof}
($\Longleftarrow$) If $C(C_4, f) \cong C(C_4, id)$, then $\gamma(C_4)=2=\gamma(C(C_4, id))$ by Theorem
\ref{C_n identity}.\\

\noindent ($\Longrightarrow$) Let $\gamma(C(C_n, f))=\gamma(C_n)$
for $n \ge 3$. By Corollary \ref{cor on permutation}, $n=4$. If
$f$ is a permutation, then $C(C_4, f)$ is isomorphic to the graph
(A) or (B) in Figure \ref{C4} (refer to \cite{CH, functi} for details).
\begin{figure}[htbp]
\begin{center}
\scalebox{0.5}{\input{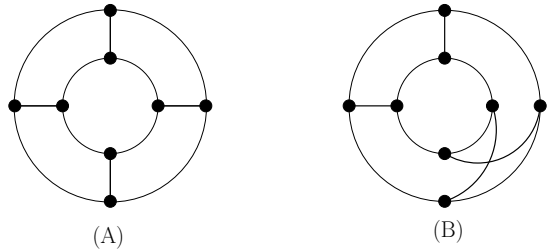}} \caption{Two
non-isomorphic graphs of $C(C_4, f)$ for a permutation $f$}
\label{C4}
\end{center}
\end{figure}
If $C(C_4, f) \cong C(C_4, id)$, then we are done. If $C(C_4, f)$ is
as in (B) of Figure \ref{C4}, we claim that $\gamma(C(C_4, f)) \ge
3$.

\vspace{.1in}

Since $|V(C(C_4, f))|=8$ and $C(C_4, f)$ is 3-regular, $D=\{w_1,
w_2\}$ dominates $C(C_4, f)$ only if no vertex in $C(C_4, f)$ is
dominated by both $w_1$ and $w_2$. It suffices to consider two
cases, using the fact that $C(C_4, f) \cong C(C_4, f^{-1})$.

\vspace{.05in}

(i) $D=\{w_1, w_2\} \subseteq V(G_1)$, 

\vspace{.05in}

(ii) $w_1 \in V(G_1)$ and $w_2 \in V(G_2)$.

\vspace{.1in}

Also, we only need to consider $w_1$ and $w_2$ such that $w_1w_2 \not\in
E(C(C_4, f))$. By symmetry, there is only one specific
case to check in case (i). In case (ii), by fixing a vertex in $V(G_1)$, we see
that there are three cases to check. In each case, for any
$D=\{w_1,w_2\}$, $N[w_1] \cap N[w_2] \neq \emptyset$. Thus
$\gamma(C(C_4, f)) > 2$.~\hfill
\end{proof}

%%%%%%%%%%%%%%%%%%%%%%%%%%%%%%%%%%%%%
%%%%%%%%%%%%%%%%%%%%%%%%%%%%%%%%%%%%%

\section{Upper Bound of $\gamma(C(C_n,f))$}

In this section we investigate domination number of functigraphs
for cycles: We show that $\gamma(C(C_n, f)) < 2 \gamma(C_n)$ for
$n \equiv 1, 2 \pmod 3$. For $n \equiv 0 \pmod 3$, we characterize
the domination number for an infinite class of functions and state
conditions under which the upper bound is not achieved. Our result
in this section generalizes a result of Burger, Mynhardt, and
Weakley in \cite{Mynhardt} which states that no cycle other than
$C_3$ and $C_6$ is a \textit{universal doubler} (i.e., only for
$n=3,6$, $\gamma(C(C_n, f)) =2\gamma(C_n)$ for any permutation
$f$).

\subsection{A characterization of $\gamma(C(C_{3k+1}, f))$}

\begin{prop}\label{n=1 mod 3}
For any function $f$, $\gamma(C(C_{3k+1},f)) < 2 \gamma(C_{3k+1})$
for $k \in \mathbb{Z}^+$.
\end{prop}

\begin{proof}
Without loss of generality, we may assume that $u_1v_1 \in E(C(C_n, f))$. Since $D=
\{v_1\} \cup \{u_{3j}, v_{3j} \ | \ 1 \le j \le k\}$ is a
dominating set of $C(C_{3k+1},f)$ with $|D|=2k+1$ for any function
$f$, $\gamma(C(C_{3k+1},f)) < 2 \gamma(C_{3k+1})$ for $k \in
\mathbb{Z}^+$. \hfill
\end{proof}

\subsection{A characterization of $\gamma(C(C_{3k+2}, f))$}

We begin with the following example showing $\gamma(C(C_5, f)) < 2
\gamma(C_5)$ for any function $f$.\\

\begin{exm} \label{ex1}
For any function $f$, $\gamma(C(C_5,
f)) < 2 \gamma(C_5)$.
\end{exm}

\begin{proof}
Let $G=C_5$, $V(G_1)=\{1, 2, 3, 4 , 5 \}$, and $V(G_2)=\{1', 2',
3', 4', 5'\}$. If $|Range(f)| \le 2$, we can choose a dominating
set consisting of all vertices in the range and, if necessary, an
additional vertex. If $|Range(f)|=3$, then we can choose the range
as a dominating set.\\

So, let $|Range(f)| \ge 4$. Then $f$ is bijective on at least
three vertices in the domain and their image. By the pigeonhole
principle, there exist two adjacent vertices, say $1$ and $2$, on
which $f$ is bijective. Let $f(1)=1'$. Then, by relabeling if
necessary, $f(2)=2'$ or $f(2)=3'$. Suppose $f(2)=3'$. Then
$D=\{1', 3', 4\}$ forms a dominating set, and we are done. Suppose
then $f(2)=2'$. We consider two cases.\\

\textit{Case 1. $|Range(f)|=4$: }By symmetry, $5' \not\in
Range(f)$ is the same as $3' \not\in Range(f)$. So, consider two
distinct cases, $5' \not\in Range(f)$ and $4' \not\in Range(f)$.
If $5' \not\in Range(f)$, then $D=\{1, 3', 4' \}$ forms a
dominating set. If $4' \not\in Range(f)$, then $D=\{1, 3', 5'\}$
forms a dominating set. In either case, we have $\gamma(C(C_5, f))
< 2 \gamma(C_5)$.\\

\textit{Case 2. $f$ is a bijection (permutation): } Recall
$f(1)=1'$ and $f(2)=2'$; there are thus 3!=6 permutations to
consider. Using the standard cycle notation, the permutations are
$(3,4)$, $(3,5)$, $(4,5)$, $(3,4,5)$, $(3,5,4)$, and identity.
However, they induce only four non-isomorphic graphs, since
$(3,4)$ and $(4,5)$ induce isomorphic graphs and $(3,4,5)$ and
$(3,5,4)$ induce isomorphic graphs. If $f$ is either $(3,4)$ or
$(3,4,5)$, then $D=\{2, 3', 5'\}$ is a dominating set. If $f$ is
$(3,5)$, then $D=\{1', 3, 3'\}$ is a dominating set. When $f$ is
the identify function, $D=\{1', 3, 5'\}$ is a dominating set. It is
thus verified that $\gamma(C(C_5, f)) < 2 \gamma(C_5)$. \hfill
\end{proof}

\begin{rem}
Example \ref{ex1} has the following
implication. Given $C(C_{3k+2},f)$ for $k \in \mathbb{Z}^+$,
suppose there exist five consecutive vertices being mapped by $f$
into five consecutive vertices. Then
$\gamma(C(C_{3k+2},f))<2\gamma(C_{3k+2})=2k+2$, and here is a proof. Relabeling if necessary, we may assume that $\{u_1,u_2,u_3,u_4,u_5\}$ are mapped into $\{v_1,v_2,v_3,v_4,v_5\}$; let $S=\{u_i, v_i \mid 1\leq i\leq 5\}$. Then $\langle S \rangle$ in $C(C_{3k+2},f)$ and the additional edge set $\{u_1u_5, v_1v_5\}$ form a graph isomorphic to a $C(C_5,f)$, which has a dominating set $S_0$ with $|S_0| \leq 3$.  In $C(C_{3k+2},f)$, if $S$ is dominated by $S_0$, then $D=S_0 \cup \{u_{3j+1} \mid 2\leq j\leq k\} \cup \{v_{3j+1} \mid 2\leq j \leq k\}$ forms a dominating set for $C(C_{3k+2},f)$ with at most $2k+1$ vertices. If $u_1$ is not dominated by $S_0$ in $C(C_{3k+2},f)$, then it is dominated solely by $u_5$ of $S_0$ in $C(C_5,f)$.  But then $u_6$ is dominated by $u_5$ in $C(C_{3k+2},f)$ and we can replace $\{u_{3j+1} \mid 2 \leq j \leq k\}$ with $\{u_{3j+2} \mid 2 \leq j \leq k\}$ to form $D$.  Similarly, if $u_5$ is not dominated by $S_0$ in $C(C_{3k+2},f)$, then it is dominated solely by $u_1$ of $S_0$ in $C(C_5,f)$.  Then $u_{3k+2}$ is dominated by $u_1$ in $C(C_{3k+2},f)$ and we can replace $\{u_{3j+1} \mid 2 \leq j \leq k\}$ with $\{u_{3j} \mid 2 \leq j \leq k\}$ to form $D$.  The cases where $v_1$ or $v_5$ is not dominated by $S_0$ in $C(C_{3k+2},f)$ can be likewise handled. Thus, if five consecutive vertices are mapped by $f$ into five consecutive vertices, then $\gamma(C(C_{3k+2},f)) \leq 2k+1 < 2k+2 = 2\gamma(C_{3k+2})$.
\end{rem}

\begin{rem}
Unlike $C(C_5,f)$, it is easily checked that $\gamma(C(P_5,
f))=2\gamma(P_5)$ for the function $f$ given in Figure
\ref{path=4}, where $P_5$ is the path on five vertices.
\end{rem}

\begin{figure}[htbp]
\begin{center}
\scalebox{0.45}{\input{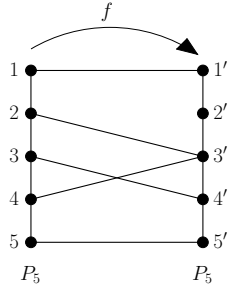}} \caption{An example where
$\gamma(C(P_5, f))=2\gamma(P_5)$}\label{path=4}
\end{center}
\end{figure}

Now we consider the domination number of $C(C_{3k+2}, f)$ for a
non-permutation function $f$, where $k \in \mathbb{Z}^+$.

\begin{theorem}\label{3k+2 non permutation}
Let $f: V(C_{3k+2}) \rightarrow V(C_{3k+2})$ be a function which
is not a permutation. Then $\gamma(C(C_{3k+2},f)) < 2
\gamma(C_{3k+2})=2k+2$.
\end{theorem}

\begin{proof}
Suppose $f$ is a function from $C_{3k+2}$ to $C_{3k+2}$ and $f$ is
not a permutation.  There must be a vertex $v_1$ in $G_2$ such
that $\deg(v_1) \geq 4$ in $C(C_{3k+2},f)$. Define the sets $V_1 =
\{v_{3i+1} \ | \ 0\leq i \leq k\}$, $V_2 = \{v_{3i+2} \ | \ 0 \leq
i \leq k\}$, and $V_3=\{v_{3i} \ | \ 1 \leq i \leq k\} \cup
\{v_1\}$.  Notice that each of these three sets is a minimum
dominating set of $G_2$ of cardinality $k+1$. Also, notice that
$|f^{-1}(V_1)|+|f^{-1}(V_2)|+|f^{-1}(V_3)|$ counts every vertex in
the pre-image of $V(G_2) \setminus \{v_1\}$ once and every vertex in the
pre-image of $\{v_1\}$ twice, so
$|f^{-1}(V_1)|+|f^{-1}(V_2)|+|f^{-1}(V_3)| \geq 3k+4$.  By the
Pigeonhole Principle, $|f^{-1}(V_i)| \geq
\lceil\frac{3k+4}{3}\rceil = k+2$ for some $i$.  Set $D_2 = V_i$
for this $i$ and notice that $D_2$ is a dominating set of $G_2$
with cardinality $k+1$ and $|f^{-1}(D_2)| \geq k+2$.\\

Without loss of generality, we may assume that $u_1$ is in $f^{-1}(D_2)$. If there
exists $0 \leq i \leq k$ such that $u_{3i+2}$ is also in the
pre-image of $D_2$, then $D_1 = \{u_{3j} \ | \ 1 \leq j \leq i\}
\cup \{u_{3j+1} \ | \ i+1 \leq j \leq k\}$ dominates the remaining
vertices of $G_1$.  Otherwise, there are at least $k+1$ vertices
in $f^{-1}(D_2)\cap\{u_{3j},u_{3j+1} \ | \ 1 \leq j \leq k\}$.  By
the Pigeonhole Principle, there exist two vertices $u_{3j_0}$ and
$u_{3j_0+1}$ in $f^{-1}(D_2)$ which are adjacent in $G_1$. Then $D_1 =
\{u_1\} \cup \{u_{3j+1} \ | \ 1 \leq j \leq j_0-1\} \cup \{u_{3j'}  \mid j_0+1 \le j' \le k\}$ dominates the
remaining vertices of $G_1$.  In either case, $D_1 \cup D_2$ is a
dominating set of $C(C_{3k+2},f)$ with $2k+1$ vertices.~\hfill
\end{proof}

For $G_i \subseteq C(G,f)$ ($i=1,2$), the distance between $x$ and
$y$ in $\langle V(G_i) \rangle$ is denoted by $d_{G_i}(x,y)$.

\begin{theorem}\label{3k+2 distance}
Let $f:V(C_{3k+2}) \rightarrow V(C_{3k+2})$ be a function, where
$k \in \mathbb{Z}^+$. For the cycle $C_{3k+2}$, if there exist two
vertices $x$ and $y$ in $G_1$ such that $d_{G_1}(x,y) \equiv 1\hspace*{-.1in} \pmod 3$ and $d_{G_2}(f(x),f(y)) \not\equiv 1 \hspace*{-.1in} \pmod 3$, 
then $\gamma(C(C_{3k+2}, f))$ $< 2 \gamma(C_{3k+2})$.
\end{theorem}

\begin{proof}
Let $x=1$ and $y=3a+2$ for a nonnegative integer $a$. By
relabeling, if necessary, we may assume that $f(x)=1'$. Note that
$D_1=(\cup_{i=1}^{a}\{3i\}) \cup (\cup_{i=a+1}^{k}\{3i+1\})$
dominates vertices in $V(G_1) \setminus \{x, y\}$. If $f(x)=1'=f(y)$, let
$D_2$ be any minimum dominating set of $G_2$ containing $1'$. Then
$D=D_1 \cup D_2$ is a dominating set of $C(C_{3k+2},f)$ with $|D|
\le 2k+1$. Thus, we assume that $f(x) \neq f(y)$. Since
$d_{G_2}(f(x),f(y)) \not\equiv 1 \pmod 3$, $f(y)=(3\ell)'$ or
$f(y)=(3\ell+1)'$ for some $\ell$ ($1 \le \ell \le k$). First,
consider when $\ell >1$. If $f(y)=(3 \ell)'$, let
$D_2=(\cup_{i=1}^{\ell-1} \{(3i+1)'\}) \cup (\cup_{i= \ell+1}^{k}
\{(3i)'\}) \cup \{1', (3 \ell)'\}$; and if $f(y)=(3 \ell+1)'$, let
$D_2=(\cup_{i=1}^{\ell-1} \{(3i+1)'\}) \cup (\cup_{i= \ell+1}^{k}
\{(3i+1)'\}) \cup \{1', (3 \ell+1)'\}$. Second, consider when
$\ell=1$. If $f(y)=(3 \ell)'$, let $D_2=(\cup_{i=1}^k\{(3i)'\})
\cup \{1'\}$; if $f(y)=(3 \ell+1)'$, let
$D_2=(\cup_{i=1}^{k}\{(3i+1)'\}) \cup \{1'\}$. Notice that $D_2$
dominates $V(G_2) \cup \{x,y\}$ in each case. Thus $D=D_1 \cup
D_2$ is a dominating set of $C(C_{3k+2}, f)$ with
$|D|=|D_1|+|D_2| =k+k+1=2k+1 < 2 \gamma(C_{3k+2}) = 2k+2$.~\hfill
\end{proof}

Next we consider $C(C_{3k+2}, f)$ for a permutation $f$.

\begin{lemma}\label{Ui}
Let $f$ be a monotone increasing function from $S=\{1,2,\ldots,
n\}$ to $\mathbb{Z}$ such that $f(1)=1$. If $|j-i| \equiv 1 \pmod
3$ implies $|f(j)-f(i)| \equiv 1 \pmod 3$ for any $i,j \in S$,
then $f(i)\equiv i \pmod 3$.
\end{lemma}

\begin{proof}
The monotonicity of $f$ -- and the rest of the hypotheses --
provides that $f(i+1)-f(i)\equiv 1 \pmod 3$, for each $1\leq i <
n$; apply it inductively to reach the conclusion. \hfill
\end{proof}

\begin{theorem}\label{3k+2 permutation}
Let $G=C_{3k+2}$ for a positive integer $k$, and let $f: V(G_1)
\rightarrow V(G_2)$ be a permutation, where the vertices in both
the domain and codomain are labeled 1 through $3k+2$. Assume
\begin{equation}\label{*}
d_{G_2}(f(x), f(y)) \equiv 1 \pmod 3 \mbox{ whenever }d_{G_1}(x,y)
\equiv 1 \pmod 3.
\end{equation}
If $f(1)=1$, then $C(C_{3k+2},f) \cong C_{3k+2} \times K_2$.
\end{theorem}

\begin{proof}
Denote by $F(n)$ the sequence of inequalities
$f(1)<f(2)<\cdots<f(n-1)<f(n)$. By cyclically relabeling
(equivalent to going to an isomorphic graph) if necessary, we may
assume $F(3)$; now the graph $C(C_{3k+2},f)$, along with the
labeling of all its vertices, is fixed. Without loss of generality, let $f(1)=1$,
$f(2)=3y_0+2$, and $f(3)=3z_0+3$ for $0 \le y_0 \le z_0 < k$.
Notice $|x-y| \equiv 1 \pmod 3$ if and only if $d_G(x,y) \equiv 1
\pmod 3$ for $G=C_{3k+2}$; we will use $|\cdot|$ in distance
considerations. We will prove that $f$ is monotone increasing on
vertices in $G_1$ (and hence $f$ is the identity function) in two
steps: Step I is the extension to $F(5)$ from $F(3)$. Step II is
the extension to $F(3(m+1)+2)$ from $F(3m+2)$ if $1\leq m\leq
k-1$.\\

Step I. Suppose for the sake of contradiction that $F(5)$ is
false. We first prove $F(4)$ and then $F(5)$.\\

Suppose $f(4) < f(3)$. This means, by condition~(\ref{*}), that
$f(4) \equiv 2 \pmod 3$. If $f(5) < f(4)$, then condition
(\ref{*}) implies $f(5) \equiv 1 \pmod 3$. If $f(5)> f(4)$, then
condition~(\ref{*}) implies $f(5) \equiv 0 \pmod 3$. Now notice
$|1-5|\equiv 1 \pmod 3$. If $f(5)<f(4)$, then $|f(1)-
f(5)|=f(5)-f(1) \equiv 0 \pmod 3$; if $f(5)>f(4)$, then $|f(1)-
f(5)|=f(5)-f(1) \equiv 2 \pmod 3$. In either case,
condition~(\ref{*}) is violated. Thus $f(3)<f(4)$, and $f(4)
\equiv 1 \pmod 3$.

\vspace{.1in}

Suppose $f(5) < f(4)$. This means, by condition~(\ref{*}), that
$f(5) \equiv 0 \pmod 3$. Then $|f(1)- f(5)|=f(5)-f(1) \equiv 2
\pmod 3$, which contradicts condition~(\ref{*}) since, again,
$|1-5|\equiv 1 \pmod 3$. Thus we have $f(4)<f(5)$, and $f(5)
\equiv 2 \pmod 3$.\\

Step II. Suppose $F(3m+2)$ for $1\leq m\leq k-1$; we will show
$F(3(m+1)+2)$. Observe that 
\begin{equation}\label{maltese}
f(3m+5)-f(1) \equiv 1 \pmod 3 \mbox{  implies  } f(3m+5) \equiv 2 \pmod 3 .
\end{equation}

First, assume $f(3m+3)<f(3m+2)$: This means, by
condition~(\ref{*}) and Lemma \ref{Ui}, that $f(3m+3) \equiv 1
\pmod 3$. Assuming $f(3m+4)>f(3m+3)$, then $f(3m+4)\equiv 2 \pmod
3$; which in turn implies that $f(3m+5)\equiv 0 \mbox{ or } 1
\pmod 3$, either way a contradiction to (\ref{maltese}). Assuming
$f(3m+4)<f(3m+3)$, then $f(3m+4)\equiv 0 \pmod 3$; however,
comparing with $f(3)$, $f(3m+4)\equiv 1 \mbox{ or } 2 \pmod 3$,
either way a contradiction again. We have thus shown that
$f(3m+3)>f(3m+2)$, which means $f(3m+3) \equiv 0 \pmod 3$.

\vspace{.1in}

Second, assume $f(3m+4)<f(3m+3)$: This means, by
condition~(\ref{*}) and Lemma \ref{Ui}, that $f(3m+4) \equiv 2
\pmod 3$. Assuming $f(3m+5)>f(3m+4)$, we have $f(3m+5)\equiv 0
\pmod 3$. Assuming $f(3m+5)<f(3m+4)$, we have $f(3m+5)\equiv 1
\pmod 3$. Either way we reach a contradiction to (\ref{maltese}). We
have thus shown that $f(3m+4)>f(3m+3)$, which means $f(3m+4)
\equiv 1 \pmod 3$.

\vspace{.1in}

Finally, assume $f(3m+5)<f(3m+4)$: This means, by
condition~(\ref{*}) and Lemma \ref{Ui}, that $f(3m+5) \equiv 0
\pmod 3$, which is a contradiction to (\ref{maltese}). Thus, $f(3m+5)>
f(3m+4)$ and $f(3m+5) \equiv 2 \pmod 3$. \hfill
\end{proof}

\begin{theorem}
For any function $f$, $\gamma(C(C_{3k+2}, f)) < 2
\gamma(C_{3k+2})$, where $k \in \mathbb{Z}^+$.
\end{theorem}

\begin{proof}
Combine Theorem \ref{C_n identity}, Theorem \ref{3k+2 non
permutation}, Theorem \ref{3k+2 distance}, and Theorem \ref{3k+2
permutation}.~\hfill
\end{proof}

\subsection{Towards a characterization of $\gamma(C(C_{3k}, f))$}

\begin{dnt}
Let $f$ be a function from $S=\{1,2,\ldots,3k\}$ to itself. We say
$f$ is a \textbf{three-translate} if $f(x+3i)=f(x)+3i$ for
$x\in\{1,2,3\}$ and $i\in\{0,1,\ldots,k-1\}$. Let
$\widetilde{f}=f|_{\{1,2,3\}}$.
\end{dnt}

\noindent \textbf{Notation}. Denote by $\widetilde{f}=(a_1, a_2,
a_3)$ the function such that $\widetilde{f}(1)=a_1$,
$\widetilde{f}(2)=a_2$, and $\widetilde{f}(3)=a_3$. We use
$C(C_{3k},f)$ and $C(C_{3k}, \widetilde{f})$ interchangeably when
$f$
is a three-translate.\\

First consider $C(C_{3k}, f)$ for a three-translate permutation
$f$.

\begin{figure}[htbp]
\begin{center}
\scalebox{0.43}{\input{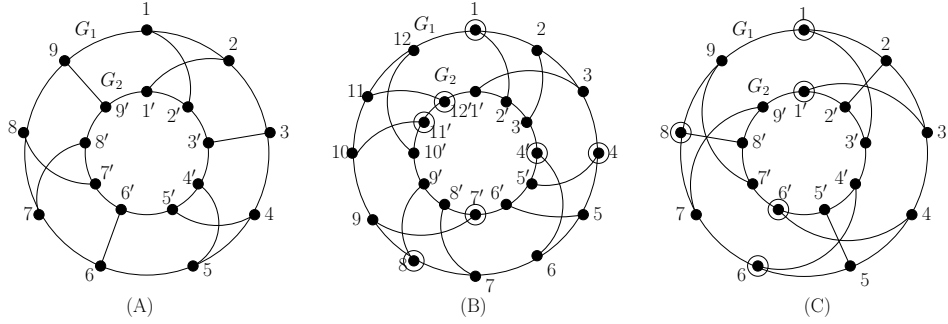}} \caption{Examples of
$C(C_{3k},f)$ for three-translate permutations $f$ when $k \ge 3$}
\label{3TP}
\end{center}
\end{figure}

\begin{theorem}\label{three-translate permutation}
Let $f$ be a three-translate permutation and let $k \ge 4$. Then
$\gamma(C(C_{3k}, f))=2k=2 \gamma(C_{3k})$ if and only if
$\widetilde{f}$ is $(2,1,3)$ or $(1,3,2)$.
\end{theorem}

\begin{proof}
Notice that $\widetilde{f}$ is one of the six permutations: identity,
$(1,3,2)$, $(2,1,3)$, $(2,3,1)$, $(3,1,2)$, and $(3,2,1)$. First,
the identity does not attain the upper bound for $k \ge 3$ by
Corollary \ref{id}. Second, the permutations $(2,3,1)$ and
$(3,1,2)$ are inverses of each other and induce isomorphic graphs in
$C(C_{3k},f)$; they do not attain the upper bound for $k \ge 4$:
$D=\{1,4,8,4',7',11',12'\}$ is a dominating set of $C(C_{12},f)$
where $\widetilde{f}=(2,3,1)$ (see (B) of Figure \ref{3TP}).
Third, the transposition $(3,2,1)$ fails to attain the upper bound
for $k \ge 3$: $D=\{1,6,8,1',6'\}$ is a dominating set of $C(C_9,
f)$ (see (C) of Figure \ref{3TP}). When $\widetilde{f}$ is
$(2,3,1)$ or $(3,1,2)$ or $(3,2,1)$, one can readily see how to
extend a dominating set from $k$ to $k+1$. Lastly, the
transpositions $(1,3,2)$ and $(2,1,3)$ induce isomorphic graphs in
$C(C_{3k}, f)$.\\

Claim: If $\widetilde{f}$ is $(1,3,2)$ or $(2,1,3)$, then
$\gamma(C(C_{3k}, f))=2k=2 \gamma(C_{3k})$ for each $k \ge 3$.\\

For definiteness, let $\widetilde{f}=(2,1,3)$ (see
(A) of Figure \ref{3TP}). For the sake of contradiction, assume
$\gamma(C(C_{3k},f)) < 2\gamma(C_{3k})=2k$ and consider a minimum
dominating set $D$ for $C(C_{3k},f)$.  We can partition the
vertices into $k$ sets $S_i=\{u_{3i-2}, u_{3i-1}, u_{3i},
v_{3i-2}, v_{3i-1}, v_{3i}\}$ for $1 \leq i \leq k$.  By the
Pigeonhole Principle, $|D \cap S_i| \leq 1$ for some $i$. Without loss of generality, we
assume that $|D \cap S_1| \leq 1$. Since neither $u_2$ nor $v_2$
has a neighbor that is not in $S_1$, $D \cap S_1$ must be either
$\{u_1\}$ or $\{v_1\}$ -- in order for both $u_2$ and $v_2$ to be
dominated by only one vertex.

\vspace{.1in}

Notice that $u_3$ and $v_3$ are dominated neither by $u_1$ nor by
$v_1$, so $D \cap S_2$ must contain both $u_4$ and $v_4$. But then
either $|D \cap S_2| \geq 3$ or $u_6$ and $v_6$ are not dominated
by any vertex in $D \cap S_2$: if $|D \cap S_2| \ge 3$, we start
the argument anew at $S_3$; thus we may, without loss of generality, assume $u_6$ and
$v_6$ are not dominated by any vertex in $D \cap S_2$ and $|D \cap
S_2| =2$. This forces $u_7$ and $v_7$ to be in $D$, but this still
leaves $u_9$ and $v_9$ un-dominated by any vertex in
$\cup_{i=1}^{3} (D \cap S_i)$. Again, if $|D \cap S_3| \ge 3$, we
start the argument anew at $S_4$. Thus, we may assume $u_9$ and
$v_9$ are not dominated by any vertex in $\cup_{i=1}^{3}(D \cap
S_i)$.

\vspace{.1in}

This pattern (allowing restarts) is forced to persist if
$\gamma(C(C_{3k}, f)) <  2k$. Now, one of two situations prevails
for $U_k$: First, the argument begins anew at $U_k$. In this case,
even if $u_{3k-2}$ and $v_{3k-2}$ are dominated by vertices
outside $S_k$, one still has $|D \cap S_k| \ge 2$, and hence $|D|
\ge 2k$. Second, the vertices $u_{3k-2}$ and $v_{3k-2}$ are
already in $D$. And if $|D \cap S_k|=2$, then either $u_{3k}$ or
$v_{3k}$ is left un-dominated. Therefore, $|D \cap S_k| \ge 3$;
this means $|D| \ge 2k$, contradicting the original
hypothesis.~\hfill
\end{proof}

\begin{rem}
One can readily check that $\gamma(C(C_{12k}, (2,3,1))) =\gamma(C(C_{12k}, (3,1,2))) \le 7k$
and $\gamma(C(C_{9k}, (3,2,1))) \le 5k$ for $k \in \mathbb{Z}^+$. 
\end{rem}

Next we consider $C(C_{3k}, f)$ for a non-permutation
three-translate $f$. Note that constant three-translates (i.e., $\tilde{f} =$ constant) never achieve the upper bound.\\

\begin{rem}\label{nonperm remark}
For $k \ge 3$, it is easy to check that there are five
non-isomorphic and non-constant three-translates which are not permutations. That is, (i) $C(C_{3k},(1,1,2)) \cong C(C_{3k},(1,1,3)) \cong C(C_{3k},(1,2,2)) \cong C(C_{3k},(2,2,3)) \cong C(C_{3k},(1,3,3)) \cong C(C_{3k},(2,3,3))$; 
(ii) $C(C_{3k},(1,2,1)) \!\cong\!C(C_{3k},(2,1,2)) \!\cong\! C(C_{3k},(2,3,2))
\!\cong\! C(C_{3k},(3,2,3))$; (iii) $\!C(C_{3k},(2,1,1))\!$ 

$\!\cong\! C(C_{3k},(2,2,1)) \!\cong C(C_{3k},(3,2,2)) \!\cong\!
C(C_{3k},(3,3,2))$; (iv) $C(C_{3k},(1,3,1)) \!\cong\! C(C_{3k},(3,1,3))$; 
(v) $C(C_{3k},(3,1,1)) \cong C(C_{3k},(3,3,1))$.
\end{rem}

\begin{figure}[htbp]
\begin{center}
\scalebox{0.5}{\input{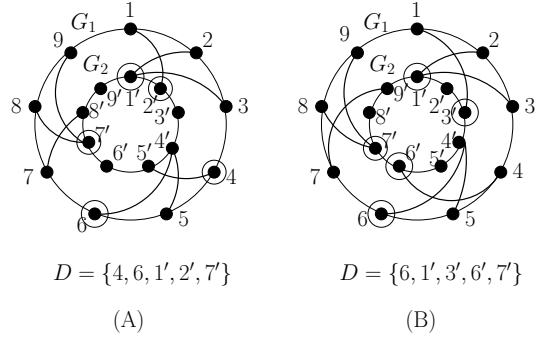}} \caption{Examples of
$C(C_{3k},f)$ such that $\gamma(C(C_{3k}, f))<2\gamma(C_{3k})$ for
non-permutation three-translates $f$ and for $k \ge
3$}\label{3T3K}
\end{center}
\end{figure}

\begin{theorem}\label{three-translate non permutation}
Let $f$ be a three-translate which is not a permutation and let $k
\ge 3$. Then $\gamma(C(C_{3k}, \widetilde{f}))=2k=2
\gamma(C_{3k})$ if and only if $C(C_{3k},\widetilde{f}) \cong
C(C_{3k},\!(1,1,2))$ or $C(C_{3k},\widetilde{f}) \cong
C(C_{3k},\!(1,2,1))$ or $C(C_{3k}, \widetilde{f}) \cong C(C_{3k},\!(1,3,1))$.
\end{theorem}

\begin{proof}
There are 21 functions which are not permutations from $S=\{1,2,
3\}$ to itself. The three constant functions obviously fail to
achieve the upper bound (if $\widetilde{f} \equiv$ constant, then $\gamma(C(C_{3k}, \widetilde{f}))=\gamma(C_{3k})=k$);
so there are 18 non-permutation functions to consider. By Remark~\ref{nonperm remark},
we need to consider five non-isomorphic classes.\\

First, we consider when the domination number of $C(C_{3k}, f)$
is less than $2\gamma(C_{3k})=2k$. If $C(C_{3k}, \widetilde{f})
\cong C(C_{3k},(2,1,1))$, then $D=\{4,6,1',2',7'\}$ is a
dominating set of $C(C_9, (2,1,1))$ (see (A) of Figure
\ref{3T3K}). If $C(C_{3k}, \widetilde{f}) \cong
C(C_{3k},(3,1,1))$, then $D\!=\!\{6,1',3',6',7'\}$ is a dominating set
of $C(C_9, (3,1,1))$ (see (B) of Figure \ref{3T3K}). In each
case, $|D|=5<2 \gamma(C_9)$, and one can readily see how to extend
a dominating set from $k$ to $k+1$ such that $\gamma(C(C_{3k},
\widetilde{f}))<2 \gamma(C_{3k})=2k$.\\

Second, we consider when $C(C_{3k}, \widetilde{f}) \cong
C(C_{3k},(1,1,2))$ or $C(C_{3k}, \widetilde{f}) \cong
C(C_{3k},(1,2,1))$ or $C(C_{3k}, \widetilde{f}) \cong
C(C_{3k},(1,3,1))$ (see Figure \ref{NPF3k}). In all three cases, $\gamma(C(C_{3k}, \widetilde{f})) = 2 \gamma(C_{3k})$ and our proofs for the three cases agree in the main idea but differ in details.\\

Here is the main idea. Since one can explicitly check the few cases when $k<3$, assume $k\geq 3$. In all three cases, we view $C(C_{3k}, \widetilde{f})$ as the union of $k$ subgraphs $\langle U_i \rangle$ for $1\leq i\leq k$, where $U_i=\{u_{3i-2}, u_{3i-1}, u_{3i}, v_{3i-2}, v_{3i-1}, v_{3i}\}$, together with two additional edges between $U_i$ and $U_j$ exactly when $i-j\equiv -1\mbox{\,or\,} 1\!\!\pmod {k}$. For each $i$, the presence of internal vertices in $U_i$ (vertices which can not be dominated from outside of $U_i$) imply the inequality $|D\cap U_i| \geq 1$. Assuming, for the sake of contradiction, that there exists a minimum dominating set $D$ with $|D|<2k$, we conclude, by the pigeonhole principle, the existence of a ``deficient $U_p$" (i.e., $|D\cap U_p|=1<2$). Starting at this $U_p$ and sequentially going through each $U_i$, we can argue that this deficient $U_p$ is necessarily compensated (or ``paired off") by an ``excessive $U_q$" (i.e., $|D\cap U_q|>2$). Going through all indices in $\{1,2, \ldots, k\}$, we are forced to conclude that $|D|\geq 2k$, contradicting our hypothesis. To avoid undue repetitiveness, we provide a detailed proof only in one of the three cases, the case of $C(C_{3k},(1,3,1))$, which is isomorphic to $C(C_{3k},(3,1,3))$.\\

Claim: If $C(C_{3k}, \widetilde{f}) \cong C(C_{3k},(3,1,3))$, then $\gamma(C(C_{3k}, f))=2k=2 \gamma(C_{3k})$.\\

\emph{Proof of Claim.} The assertion may be explicitly verified for $k<4$; so let $k\geq 4$.
For the sake of contradiction, assume $\gamma(C(C_{3k},f)) < 2k$
and consider a minimum dominating set $D$ for $C(C_{3k},f)$. We
can partition the vertices into $k$ sets $U_i=\{u_{3i-2}, u_{3i-1}, u_{3i}, v_{3i-2}, v_{3i-1}, v_{3i}\}$ for $1 \leq i \leq
k$.  By the Pigeonhole Principle, $|D \cap U_i| \leq 1$ for some
$i$. Without loss of generality, we assume that $|D \cap U_1| \leq 1$. Since neither
$u_2$ nor $v_2$ has a neighbor that is not in $U_1$, $D \cap U_1$
must be $\{v_1\}$ -- the only vertex to dominate both $u_2$ and $v_2$.\\

Notice that $u_3$ and $v_3$ are not dominated by $v_1$, the only
vertex in $D \cap U_1$, so $D \cap U_2$ must contain both $u_4$
and $v_4$. But then either $|D \cap U_2| \geq 3$ or $u_6$ is not
dominated by any vertex in $D \cap U_2$: if $|D \cap U_2| \geq 3$,
we start the argument anew at $U_3$; thus we may, without loss of generality, assume
$u_6$ is not dominated by any vertex in $D \cap U_2$. This forces
$u_7$, which dominates $u_6$, $u_8$, and $v_9$, to be in $D$. Now,
for $v_7$ and $v_8$ to be dominated, one of them must be in $D$.
But this still leaves $u_9$ un-dominated by any vertex in
$\cup_{i=1}^{3}U_i$. Again, if $|D \cap U_3| \geq 3$, we start the
argument anew at $U_4$. Thus, we may, without loss of generality, assume $u_9$ is not
dominated by any vertex in
$\cup_{i=1}^{3}U_i$.\\

This pattern (allowing restarts) is forced to persist if
\mbox{$\gamma(C(C_{3k},f)) < 2k$}. Now, one of two situations
prevails for $U_k$: First, the argument begins anew at $U_k$. In
this case, even if $u_{3k-2}$ and $v_{3k-2}$ are dominated by
vertices outside of $U_k$, one still has $|D \cap U_k| \geq 2$,
and hence $|D| \geq 2k$. Second, the vertices $u_{3k-2}$ and
either $v_{3k-2}$ or $v_{3k-1}$ are already in $D$. And if $|D\cap
U_k|=2$, then $u_{3k}$ (and, for that matter, $u_1$) is left
un-dominated. Therefore, $|D\cap U_k|\geq 3$ and $|D|\geq 2k$,
contradicting the original hypothesis. ~\hfill
\end{proof}

\begin{figure}[htbp]
\begin{center}
\scalebox{0.5}{\input{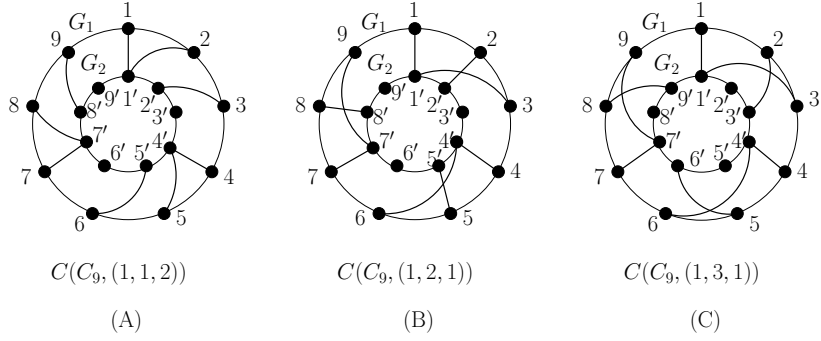}}\caption{Examples of
$C(C_{3k},f)$ such that $\gamma(C(C_{3k}, f))=2\gamma(C_{3k})$ for
non-permutation three-translates $f$ and for $k \ge 3$}
\label{NPF3k}
\end{center}
\end{figure}

\newpage
Now, we consider sufficient conditions for $\gamma(C(C_{3k}, f)) <
2 \gamma(C_{3k})$ in terms of the maximum and the average degree
of $C(C_{3k}, f)$, respectively.

\begin{prop}\label{max degree}
If $\Delta(C(C_{3k},f)) \ge k+5$, then $\gamma(C(C_{3k},f)) < 2
\gamma(C_{3k})$.
\end{prop}

\begin{proof}
Suppose $C(C_{3k},f)$ is a functigraph with maximum degree at
least $k+5$. Without loss of generality, we assume that the degree of $v_1$ is at least
$k+5$. Partition the vertices of $G_1$ into $k$ sets $U_i =
\{u_{3i-2},u_{3i-1},u_{3i}\}$, where $1 \le i \le k$. If $N[v_1]$
contains any set $U_i$, say $U_1 \subseteq N[v_1]$, then $\{u_i \
| \ i\geq 5 \mbox{ and }i \equiv 2 \pmod 3\} \cup \{v_i \ | \ i
\equiv 1 \pmod 3 \}$ is a dominating set of $C(C_{3k},f)$ with
$2k-1$ vertices. Thus, we may assume that $|N[v_1] \cap U_i|\leq
2$ for each $i$. It follows that $|N[v_1]\cap U_i|=2$ for at least
$3$ different values of $i$, say $i = p, q,$ and $r$. Let $x$,
$y$, and $z$ be the vertices in $G_1$ that are in $U_p$, $U_q$,
$U_r$ (respectively) and not in $N[v_1]$.\\

Suppose one of $x$, $y$, and $z$, say $x$, maps to a vertex
$v_{3j+1}$ for some $j$. Then $\{u_{\ell} \ | \ \ell \equiv 2
\pmod 3 \mbox{ and } \ell \neq 3p-1\} \cup \{v_{\ell} \ | \ \ell
\equiv 1 \pmod 3\}$ is a dominating set of $C(C_{3k},f)$ with
$2k-1$ vertices. Otherwise, two of $x$, $y$, and $z$, say $x$ and
$y$, map to vertices $v_{s}$ and $v_{t}$ such that $s \equiv t$
(mod 3), say $s \equiv t \equiv 0 \pmod 3$, without loss of generality. But then the
set $\{u_{\ell} \ | \ \ell \equiv 2 \pmod 3, \ell \neq 3p-1,
\mbox{ and } \ell \neq 3q-1\} \cup \{v_1\} \cup \{v_{\ell} \ | \
\ell \equiv 0 \pmod 3\}$ is a dominating set of $C(C_{3k},f)$
with $2k-1$ vertices. \hfill
\end{proof}

The following example shows that the bound provided in Proposition
\ref{max degree} is nearly sharp. Namely, there exists a function
$f:V(C_{3k})\rightarrow V(C_{3k})$ such that the resulting
functigraph has $\Delta(C(C_{3k},f))=k+3$ and
$\gamma(C(C_{3k},f))=2 \gamma(C_{3k})=2k$.\\

\begin{exm}\label{ex2}
For $k\in \mathbb{Z}^+$, let $f: V(C_{3k}) \rightarrow V(C_{3k})$ be a function defined by 
\begin{displaymath}
f(u_i)= \left\{
\begin{array}{ll}
v_{i} & \mbox{ if } i  \equiv 1 \pmod 3,\\
v_{i+1} & \mbox{ if } i \equiv 2 \pmod 3,\\
v_{3k} & \mbox{ if } i  \equiv 0 \pmod 3.
\end{array} 
\right.
\end{displaymath}
Then $\gamma(C(C_{3k},f))=2k=2\gamma(C_{3k})$.
\end{exm}

\begin{proof} Notice that $\Delta(C(C_{3k},f))=\deg(v_{3k})=k+3$. For $1\leq i\leq k$, define
$S_i=\{u_{3i},u_{3i-1},u_{3i-2},v_{3i},v_{3i-1},v_{3i-2}\}$, and notice that $\cup_{i=1}^{k}S_i$ is a partition of
$V(C(C_{3k},f))$. Let $D$ be any dominating set of $C(C_{3k},f)$;
we need to show that $|D|\geq 2k$. Observe that $|D \cap S_i|\geq 1$
since neither $u_{3i-1}$ nor $v_{3i-1}$ can be dominated from
outside of $S_i$ for $1\leq i\leq k$. We will argue in an inductive
fashion starting at $k$ and descending to $1$.\\

Suppose $|D|<2k$; choose the biggest $j\leq k$ such that $|D \cap
S_j|=1$. Of necessity $v_{3j}\in D$, as it is the only vertex in
$S_j$ dominating both $u_{3j-1}$ and $v_{3j-1}$. Then $|D \cap
S_{j-1}|\geq 2$, since to dominate $u_{3j-2}$ and $v_{3j-2}$ in
$S_j$, $D$ must contain both $u_{3j-3}$ and $v_{3j-3}$ in
$S_{j-1}$.\\

Now, if $|D \cap S_{j-1}|\geq 3$, then it is ``paired off" with
$S_j$. We will choose the biggest $\ell<j$ such that $|D \cap
S_{\ell}|=1$ and restart at $S_{\ell}$ our inductive argument. Of
course, $S_j$ may be paired off with $S_q$ where $j>q\geq 1$ and
$|D \cap S_q|\geq 3$; in this case, of necessity, $|D \cap S_p|=2$
for $j>p>q$, and we restart the argument after $S_q$ when $q>1$.
Therefore, one of the following cases must hold for $S_1$.

\vspace{.1in}

(i) $|D \cap S_1|\geq 3$: then $S_1$ may be paired off
with the least $j$ such that $|D \cap S_j|=1$, if necessary.

\vspace{.05in}

(ii) $|D \cap S_1|=2$ and every $S_j$ with $|D \cap S_j|=1$ is
paired off with $S_{q}$ such that $q<j$ and $|D \cap S_{q}|\geq 3$.

\vspace{.05in}

(iii) $|D \cap S_1|=2$ and there exists $j>1$ with $|D \cap
S_j|=1$ which is not paired off with some $S_{q}$ such that $q<j$
and $|D \cap S_{q}|\geq 3$: If $j=k$, then by examining $S_k$,
$S_{k-1}$, and $S_1$, we will readily see that the assumption is
impossible ($u_{1}$ is not dominated). If $j<k$, then there must
exist $q>j$ such that
$|D \cap S_{q}|\geq 3$ (in order to dominate $u_{3(j+1)-2}$). 

\vspace{.05in}

(iv) $|D \cap S_{1}|=1$: then there must exist $q>1$ such that $|D
\cap S_{q}|\geq 3$ (in order to dominate~$u_4$).

\vspace{.1in}

In each case, we conclude $|D|\geq 2k$, contradicting our original
supposition.\hfill
\end{proof}

\begin{prop}\label{average degree}
Suppose $C(C_{3k},f)$ is a functigraph with domain $G_1$ and
codomain $G_2$. Partition $G_2$ into three sets $V_1$, $V_2$, and
$V_3$ such that $V_i = \{v_j \ | \ j \equiv i \pmod 3\}$. If there
is some $i$ such that the average degree over all vertices in
$V_i$ is strictly greater than $4$, then $\gamma(C(C_{3k},f)) < 2
\gamma(C_{3k})$.
\end{prop}

\begin{proof}
Suppose $C(C_{3k},f)$ is a functigraph with codomain $G_2$ and
that there is some $i$, say $i=1$, such that the average degree
over all vertices in $V_1$ is strictly greater than $4$. Then
$|N[V_1] \cap V(G_1)| \geq 2k+1$. Let $U_1$ be the vertices in
$V(G_1)$ that are not in $N[V_1]$ and notice that $|U_1| \leq
k-1$. Then $U_1 \cup V_1$ is a dominating set of $C(C_{3k},f)$.
\hfill
\end{proof}

\begin{rem}
The result obtained in
Proposition~\ref{average degree} is sharp as shown in Example \ref{ex2}.
In the example, the average degree of the vertices in $V_3$ is exactly 4.
\end{rem}

\textit{Acknowledgement.} The authors wish to thank Andrew Chen
for a motivating example -- the graph (B) in Figure \ref{NPF3k}. The authors also thank the referees and the editor for corrections and suggestions, which improved the paper.

%%%%%%%%%%%%%%%%%%%%%%%%%%%%%%%%%%%%%%%%%%%%%%%%%%%%%%%
%%%%%%%%%%%%%%%%%%%%%%%%%%%%%%%%%%%%%%%%%%%%%%%%%%%%%%%

\end{document}